\documentclass{ws-procs9x6}

\DeclareMathOperator{\Rad}{Rad}
\DeclareMathOperator{\rank}{rank}
\DeclareMathOperator{\const}{const}
\DeclareMathOperator{\tr}{tr}
\DeclareMathOperator{\ltr}{ltr}
\DeclareMathOperator{\codim}{codim}
\DeclareMathOperator{\Ker}{Ker}

\begin{document}

\title{SOME LIGHTLIKE SUBMANIFOLDS OF ALMOST \\COMPLEX MANIFOLDS WITH NORDEN METRIC}

\author{Galia Nakova}

\address{University of Veliko Tarnovo "St. Cyril and St. Metodius"\\
Faculty of Mathematics and Informatics,\\
T. Tarnovski 2 str.,\\
5000 Veliko Tarnovo, Bulgaria \\ E-mail:gnakova@gmail.com}

\begin{abstract}
In this paper we study submanifolds of an almost complex manifold with Norden metric which are non-degenerate
with respect to the one Norden metric and lightlike with respect to the other Norden metric on the manifold.
Relations between the induced geometric objects of some of these submanifolds are given. 
Examples of the considered submanifolds are constructed.
\end{abstract}


\bodymatter

\section*{Introduction}
The general theory of lightlike submanifolds has been developed in \cite{D-B} by K. Duggal and A. Bejancu.
The geometry of Cauchy-Riemann (CR) lightlike submanifolds of indefinite Kaehler manifolds was presented in 
\cite{D-B} \,, too.
Many new types of lightlike submanifolds of indefinite Kaehler manifolds, indefinite Sasakian and 
indefinite quaternion Kaehler manifolds are introduced in \cite{D-S} by K. Duggal and B. Sahin. In \cite{D-B}
and  \cite{D-S} , different applications of lightlike geometry in the mathematical physics are given.
However, lightlike submanifolds of almost complex manifolds with Norden metric (or B-metric)
have not been considered yet. The study of such submanifolds is interesting because there exists a difference 
between the geometry of a $2n$-dimensional indefinite almost Hermitian manifold and the geometry of a
$2n$-dimensional almost complex manifold with Norden metric. The difference arises due to the fact that in
the first case, the almost complex structure $\overline J$ is an isometry with respect to the semi-Riemannian metric
$\overline g$ of index $2q  (0<2q<2n)$ and in the second case, $\overline J$ is an anti-isometry with respect to
the metric $\overline g$, which is necessarily of signature $(n,n)$. This property of the pair 
$(\overline J,\overline g)$ of an almost complex manifold with Norden metric $\overline M$ allows to define the tensor
field $\overline {\widetilde g}$ on $\overline M$ by $\overline {\widetilde g}(X,Y)=\overline g(\overline JX,Y)$, 
which is a Norden metric, too. 
\par
Let $M$ be a real $m$-dimensional submanifold of 
an almost complex manifold witn Norden metric $(\overline M,\overline J,\overline g,\overline {\widetilde g})$. 
The geometry of $M$ depends on the behaviour of the tangent bundle of $M$ with respect to the action 
of the almost complex structure $\overline J$ and the induced metric on $M$. 
Due to that there exist two Norden metrics $\overline g$ and $\overline {\widetilde g}$ of
$\overline M$ we can consider two induced metrics $g$ and $\widetilde g$ on $M$ by $\overline g$ and
$\overline {\widetilde g}$, respectively. For a submanifold $M$ of $\overline M$ three cases with respect to 
the induced metrics $g$ and $\widetilde g$ on $M$ are possible:
$M$ is non-degenerate with respect to both $g$ and $\widetilde g$;
$M$ is degenerate with respect to both $g$ and $\widetilde g$;
$M$ is non-degenerate with respect to $g$ (resp. $\widetilde g$) and degenerate with respect to 
$\widetilde g$ (resp. $g$).
\par
In this paper we consider mainly submanifolds of the third type. In Sections 1 and 2 we recall some preliminaries about
lightlike submanifolds of semi-Riemannian manifolds and almost complex manifolds with Norden metric, respectively.
The main results of the paper are given in Section 3. We prove that a necessary and sufficient condition for the
submanifold $(M,g)$ of $\overline M$ to be a CR-submanifold is $(M,\widetilde g)$ to be a Radical transversal lightlike submanifold of $\overline M$. We obtain also that the submanifold $(M,g)$ of $\overline M$ is generic, totally real
or Lagrangian if and only if $(M,\widetilde g)$ is a coisotropic, an isotropic or a totally lightlike submanifold of 
$\overline M$, respectively. In Section 4, in the case of totally real $(M,g)$ and isotropic $(M,\widetilde g)$, 
relations between the induced geometric objects are found. The structure
equations of these submanifolds of a Kaehler manifold with Norden metric are obtained. In the last section we consider
known matrix Lie subgroups of $GL(2;\mathbb{C})$ as examples of the submanifolds introduced in Section 3.

\section{Lightlike submanifolds of semi-Riemannian manifolds}
Follow \cite{D-B} we give some basic notions and formulas for lightlike submanifolds
of semi-Riemannian manifolds. 
\par
Let $(\overline M,\overline g)$ be a real $(m+n)$-dimensional semi-Riemannian manifold  $(m>1, \, n\geq 1)$, i.e.
$\overline g$ is a semi-Riemannian metric of constant 
index $q\in \{1,\ldots ,m+n-1\}$ and $M$ be a submanifold of $\overline M$ of codimension $n$. Denote by $g$ the 
induced tensor field on $M$ of $\overline g$ and suppose that ${\rank} g={\const}$ on $M$. 
If ${\rank} g=m$, i.e. $g$ is non-degenerate on the tangent bundle $TM$ of $M$, then $M$ is called {\it a non-degenerate submanifold}
of $\overline M$. In the case ${\rank} g<m$, i.e. $g$ is degenerate on $TM$, then 
$M$ is called 
{\it a lightlike submanifold} of $\overline M$. We will note an orthogonal direct sum by $\bot $ and a non-orthogonal direct sum by $\oplus $. The tangent space $T_xM$ and the normal space $T_xM^\bot $ of a non-degenerate submanifold
$M$ of $\overline M$ are non-degenerate and they are complementary orthogonal vector subspaces of $T_x\overline M$.
However, if $M$ is a lightlike submanifold of $\overline M$, both $T_xM$ and $T_xM^\bot $ are
degenerate orthogonal, but non-complementary subspaces of $T_x\overline M$ and there exists a subspace 
$
{\Rad} T_xM=T_xM \cap T_xM^\bot ={\Rad} T_xM^\bot ,
$
where
\[
{\Rad} T_xM=\{\xi _x \in T_xM: g(\xi _x,X)=0, \, \, \forall X\in T_xM\}.
\]
The dimension of ${\Rad} T_xM$ depends on $x\in M$.
The submanifold $M$ of $\overline M$ is said to be {\it an $r$-lightlike} ({\it an $r$-degenerate, an $r$-null}) 
{\it submanifold} if the mapping 
$$
{\Rad} TM: x\in M \longrightarrow {\Rad} T_xM,
$$
defines a smooth distribution on $M$ of ${\rank}$ \quad $r>0$ (it means ${\dim} ({\Rad} T_xM)=r$ for all $x\in M$).
${\Rad} TM$ is called the {\it Radical (lightlike, null) distribution} on $M$. 
\par
Let $S(TM)$ be a complementary distribution of ${\Rad} TM$ in $TM$ and $S(TM^\bot )$ be a complementary
vector bundle of ${\Rad} TM$ in $TM^\bot $. Both $S(TM)$ and $S(TM^\bot )$ are non-degenerate with respect to
$\overline g$ and the following decompositions are valid
\begin{equation}
TM={\Rad} TM \bot S(TM), \quad TM^\bot ={\Rad} TM \bot S(TM^\bot ).
\label{1.1}
\end{equation}
The distribution $S(TM)$ and the vector bundle $S(TM^\bot )$ are called a {\it screen distribution} and a
{\it screen transversal vector bundle} of $M$, respectively. Although $S(TM)$ is not unique, it is canonically
isomorphic to the factor vector bundle $TM/{\Rad} TM$. 
As $S(TM)$ is a non-degenerate vector subbundle of $T\overline M$ we have
\[
T\overline M=S(TM) \bot S(TM)^\bot ,
\]
where $S(TM)^\bot $ is the complementary orthogonal vector bundle of $S(TM)$ in $T\overline M$. As
$S(TM^\bot )$ is a vector subbundle of $S(TM)^\bot $ and since both are non-degenerate we have the 
following orthogonal direct decomposition 
\[
S(TM)^\bot =S(TM^\bot )\bot S(TM^\bot )^\bot .
\] 
Let ${\tr}(TM)$ and ${\ltr}(TM)$ be complementary (but never orthogonal) vector bundles to $TM$ in $T\overline M$
and to ${\Rad} TM$ in $S(TM^\bot )^\bot$, respectively. Then we have
\begin{equation}
\begin{array}{lll}
{\tr}(TM)={\ltr}(TM)\bot S(TM^\bot ); \\
T\overline M=TM \oplus {\tr}(TM) 
=S(TM)\bot S(TM^\bot )\bot ({\Rad} TM \oplus {\ltr}(TM)).
\end{array}
\label{1.2}
\end{equation}
The vector bundle ${\ltr}(TM)$ is called  {\it a lightlike transversal vector bundle} of $M$ and
${\tr}(TM)$ is called {\it a transversal vector bundle} of $M$.
There exists a local quasi-orthonormal basis \cite{D-B}of $\overline M$ along $M$:
\[
\{\xi _i, N_i, X_a, W_\alpha \},\, \, i\in \{1,\ldots,r\}, \, \, a\in \{r+1,\ldots,m\}, 
\, \,\alpha \in \{r+1,\ldots,n\},  
\]
where $\{\xi _i\}$ and $\{N _i\}$ are lightlike basises of ${\Rad} TM_{\vert U}$ and ${\ltr}(TM)_{\vert U}$, 
respectively;
$\{X _a\}$ and $\{W _\alpha \}$ are basises of $S(TM)_{\vert U}$ and $S(TM^\bot )_{\vert U}$.

The following possible four cases with respect to the dimension $m$ and codimension $n$ of $M$ and rank $r$ of
${\Rad} TM$ are studied:\\
\begin{arabiclist}[(4)]
\item {\it r-lightlike}, if \, $0<r<{\min} (m,n)$; \\
\item {\it coisotropic}, if \, $1\leq r=n<m$, \, $S(TM^\bot )=\{0\}$; \\
\item {\it isotropic}, if \, $1<r=m<n$, \, $S(TM)=\{0\}$; \\
\item {\it totally lightlike}, if \, $1<r=m=n$, \, $S(TM)=\{0\}=S(TM^\bot )$.
\end{arabiclist}
\section{Almost complex manifolds with Norden metric}\label{sec-2}
Let $(\overline M,\overline J,\overline g)$ be a $2n$-dimensional almost complex manifold with Norden metric 
 \cite{GB} , i.e. $\overline J$ is an almost complex structure and $\overline g$ is a metric on $\overline M$ such that:
\begin{equation}\label{1}
\overline J^2X=-X, \quad \overline g(\overline JX,\overline JY)=-\overline g(X,Y)
\end{equation}
for arbitrary differentiable vector fields $X, Y$ on $\overline M$. \\
The tensor field $\overline {\widetilde g}$ of type $(0,2)$ on $\overline M$ defined by 
\begin{equation}\label{2.1}
\overline {\widetilde g}(X,Y)=\overline g(\overline JX,Y)
\end{equation}
is a Norden metric on $\overline M$, too.
Both metrics $\overline g$ and $\overline {\widetilde g}$ are necessarily
of signature $(n,n)$. The metric $\overline {\widetilde g}$ is said to be an {\it associated metric} of $\overline M$.
The Levi-Civita connection of $\overline g$ is denoted by $\overline \nabla $.
The tensor field $F$ of type $(0,3)$ on $\overline M$ is defined by 
$F(X,Y,Z)=\overline g((\overline \nabla _X\overline J)Y,Z)$.
Let $\overline {\widetilde \nabla }$ be the Levi-Civita connection of $\overline {\widetilde g}$. Then
\begin{equation}\label{3}
\Phi (X,Y)=\overline {\widetilde \nabla }_XY-\overline \nabla _XY
\end{equation}
is a tensor field of type $(1,2)$ on $\overline M$. Since $\overline \nabla $ and $\overline {\widetilde \nabla }$ 
are torsion free we have $\Phi (X,Y)=\Phi (Y,X)$.
A classification of the almost complex manifolds with Norden metric with respect to the tensor $F$ is given in
\cite{GB} and eight classes are obtained. In \cite{GGM} these classes are characterized by conditions 
for the tensor $\Phi $, too.
\par
Analogously, as for an almost
Hermitian manifold, for an almost complex manifold with Norden metric $(\overline M,\overline J,\overline g,\overline {\widetilde g})$ the following important classes of non-degenerate submanifolds of $\overline M$
with respect to the induced metric $g$ (resp. $\widetilde g$) on the submanifold can be considered:
\begin{alphlist}[(e)]
\item $M$ is called a {\it holomorphic (or an invariant) submanifold} of $\overline M$ if 
$\overline J(T_xM)=T_xM, \, \forall x\in M$. The dimension $m$ of an invariant submanifold $M$ of 
$\overline M$ is necessarily even;
\item $M$ is called a {\it totally real (or an anti-invariant) submanifold} of $\overline M$ if
$\overline J(T_xM)\subseteq T_xM^\bot , \, \forall x\in M$. In this case ${\dim} M=m\leq n$;
\item A totally real submanifold $M$ is called {\it Lagrangian} if ${\dim} M=m=n$;
\item $M$ is called a {\it CR-submanifold} of $\overline M$ if $M$ is endowed with two complementary 
orthogonal distributions $D$ and $D^\bot $ satisfying the conditions:
$\overline J(D_x)=D_x$ and $\overline J(D_x^\bot )\subset T_xM^\bot $, \, $\forall x\in M$;
\item The CR-submanifold $M$ of $\overline M$ is called a {\it generic submanifold} if 
${\dim} D^\bot ={\codim} M=2n-m$.
\end{alphlist}
The CR-submanifold $M$ of $\overline M$ is called {\it non-trivial (proper)} if ${\dim} D>0$ and
${\dim} D^\bot >0$.
The holomorphic and the totally real submanifolds are particular cases of the class of CR-submanifolds.
Holomorphic submanifolds of almost complex manifolds with Norden metric were studied by K. Gribachev. In \cite{GGM_2} ,
\cite{MM1} and \cite{MM2}
hypersurfaces of Kaehler manifolds with Norden metric were considered. 
 
\section{Submanifolds of an almost complex manifold with Norden metric which are non-degenerate
with respect to the one Norden metric and lightlike with respect to the other Norden metric}
Let $(\overline M,\overline J,\overline g,\overline {\widetilde g})$ be a $2n$-dimensional almost complex 
manifold with Norden metric and $M$ be an $m$-dimensional submanifold immersed by $\varphi $ in $\overline M$.
For simplicity we identify for each $x\in M$ the tangent space
$T_xM$ with $\varphi \ast (T_xM)\subset T_{\varphi (x)}\overline M$. Let $g$ and $\widetilde g$ be two metrics of $M$.
We assume that the immersion $\varphi $ is an isometry with respect to both metrics $\overline g$ and
$\overline {\widetilde g}$ on $\overline M$ and we identify the metrics $g$ and $\widetilde g$
on $M$ with the induced metrics on the subspace $\varphi \ast (T_xM)$ of 
$\overline g$ and $\overline {\widetilde g}$, respectively. Hence, for any $x\in M$ we have
\[
g(X_x,Y_x)=\overline g(X_x,Y_x), \quad \widetilde g(X_x,Y_x)=\overline {\widetilde g}(X_x,Y_x), \quad
\forall X_x, Y_x \in T_xM.
\]  
We note that $TM=\bigcup \limits_{x \in M} T_xM$ is the tangent bundle of both submanifolds $(M,g)$ and
$(M,\widetilde g)$ of $\overline M$. 
We will denote: the normal bundle of $(M,g)$ and $(M,\widetilde g)$ by $TM^\bot $ and $TM^{\widetilde \bot }$, respectively;
an orthogonal direct sum with respect to $\overline g$ (resp. $\overline {\widetilde g}$)
by $\bot $ (resp. $\widetilde \bot $) and a non-orthogonal direct sum by $\oplus $ (resp. $\widetilde \oplus $).
\par
In this section we give an answer to the question what type is the submanifold $(M,\widetilde g)$ of 
$\overline M$ if $(M,g)$ belongs to one of the basic classes of non-degenerate submanifolds.
\begin{lemma}
Let $(V,J,\overline g,\overline {\widetilde g})$ be a $2n$-dimensional almost complex vector space with Norden
metric and $W$ be a $2p$-dimensional holomorphic subspace of $V$. Then for the induced metrics $g$ and 
$\widetilde g$ on $W$ of $\overline g$ and $\overline {\widetilde g}$ respectively, we have
\par
1) $g$ is non-degenerate iff $\widetilde g$ is non-degenerate;
\par
2) $g$ is degenerate iff $\widetilde g$ is degenerate.
\end{lemma}
\begin{proof}
1) Let $g$ be non-degenerate. We assume that $\widetilde g$ is degenerate. Then there exists 
$\xi \in W, \, \xi \neq 0$ such that $\widetilde g(\xi ,X)=0, \, \forall X\in W$. As $W$ is a holomorphic
subspace of $V$ and ${\Ker } \, J=\{0\}$ it follows that there exists $J\xi \in W, J\xi \neq 0$ such that for 
$\forall X\in W$ we have $g(J\xi ,X)=\overline g(J\xi ,X)=\overline {\widetilde g}(\xi ,X)=\widetilde g(\xi ,X)=0$.
So, we obtain that $g$ is degenerate, which is a contradiction. Analogously, we can check that if 
$\widetilde g$ is non-degenerate, then $g$ is non-degenerate, too. The truth of 2) follows from 1).
\end{proof}

\begin{theorem}
Let $(\overline M,\overline J,\overline g,\overline {\widetilde g})$ be a $2n$-dimensional
almost complex manifold with Norden metric and $M$ be a $2p$-dimensional submanifold of $\overline M$. The 
submanifold $(M,g)$ is holomorphic iff the submanifold $(M,\widetilde g)$ is holomorphic.
\end{theorem}
\begin{proof}
The proof of the theorem follows from assertion 1) of Lemma 3.1 by replacing the space $V$ and
the subspace $W$ by the tangent bundle $T\overline M$ of $\overline M$ and  the tangent bundle $TM$ of 
the submanifolds $(M,g)$ and $(M,\widetilde g)$, respectively.
\end{proof}
Radical transversal lightlike submanifolds of indefinite Kaehler manifolds are introduced in
\cite{Sahin} by Sahin. Further, we show that such submanifolds naturally arise on an almost complex manifold 
with Norden metric $\overline M$ and they are related with CR-submanifolds of $\overline M$. First, analogously
as in \cite{Sahin} we give the following   
\begin{definition}
Let $(M,g,S(TM),S(TM^\bot ))$ be a lightlike submanifold of an almost complex manifold with Norden metric
$(\overline M,\overline J,\overline g,\overline {\widetilde g})$. We say that $M$ is a {\it Radical transversal
lightlike submanifold} of $\overline M$ if the following conditions are satisfied:
\begin{equation}
\overline J({\Rad} TM)={\ltr}(TM),
\label{3.1}
\end{equation}
\begin{equation}\label{3.2}
\overline J(S(TM))=S(TM).
\end{equation}
Moreover, we say that a Radical transversal lightlike submanifold $M$ is proper if $S(TM)\neq 0$.
\end{definition}
It is important to note:
\begin{arabiclist}[(2)]
\item Taking into account that for  an isotropic and a totally lightlike submanifold $S(TM)=0$ and Definition 3.1, it
is clear that there exists no proper Radical transversal isotropic and totally lightlike submanifold of 
$\overline M$.  
\item Contrary to the case when $M$ is a Radical transversal lightlike submanifold of an indefinite Kaehler manifold \,
\cite{Sahin} , for an almost complex manifold with Norden metric $(\overline M,\overline J,\overline g,\overline {\widetilde g})$, $M$ can be an 1-lightlike Radical
transversal lightlike submanifold of $\overline M$. Indeed, let us suppose that $(M,g)$ is an 1-lightlike Radical
transversal lightlike submanifold of $(\overline M,\overline J,\overline g,\overline {\widetilde g})$. Then there
exist basises $\{\xi \}$ and $\{N\}$ of ${\Rad} TM$ and ${\ltr}(TM)$ respectively, such that $\overline g(\xi ,N)=1$.
On the other hand, (\ref{3.1}) implies that $\overline J\xi =\alpha N\in \Gamma ({\ltr}(TM))$, 
$\alpha \in F({\ltr}(TM))$. Thus, for the function $\alpha $ we obtain
$\alpha =\overline g(\xi ,\overline J\xi )$, which is not zero.
\end{arabiclist}
\begin{proposition}
Let $(M,g)$ be a Radical transversal lightlike submanifold of an almost complex manifold with Norden metric
$(\overline M,\overline J,\overline g,\overline {\widetilde g})$. Then the distribution $S(TM^\bot )$ is holomorphic
with respect to $\overline J$.
\end{proposition}
\begin{proof}
As the tangent bundle $TM$ and the transversal vector bundle ${\tr}(TM)$ of $(M,g)$ are complementary (but not orthogonal
with respect to $\overline g$) vector subbundles of $T\overline M$, for any $X\in \Gamma (TM)$ and any 
$V\in \Gamma ({\tr}(TM))$, we have
\begin{equation}
\overline JX=TX+FX ; \qquad \overline JV=tV+fV,
\label{3.3}
\end{equation}
where $TX, tV$ belong to $\Gamma (TM)$ and $FX, fV$ belong to $\Gamma ({\tr}(TM))$. Then $T$ and $f$ are endomorphisms
on $TM$ and ${\tr}(TM)$, respectively; $F$ and $t$ are transversal bundle-valued 1-form on $TM$ and tangent 
bundle-valued 1-form on ${\tr}(TM)$, respectively. According to the decompositions (\ref{1.1}), (\ref{1.2}),
we can write any $X\in \Gamma (TM)$ and any $V\in \Gamma ({\tr}(TM))$ in the following manner
\begin{equation}
X=PX+QX ; \qquad V=LV+SV,
\label{3.4}
\end{equation}
where $P$ and $Q$ are the projection morphisms of $TM$ on $S(TM)$ and ${\Rad} TM$, respectively; $L$ and $S$ are
the projection morphisms of ${\tr}(TM)$ on ${\ltr}(TM)$ and $S(TM^\bot )$, repectively. So, for any 
$W\in \Gamma S(TM^\bot)$ we have
\begin{equation}
\overline JW=tW+fW=PtW+QtW+LfW+SfW.
\label{3.5}
\end{equation}
Now, if $X\in \Gamma S(TM)$, $\xi \in \Gamma ({\Rad} TM)$, $N\in \Gamma ({\ltr}(TM))$, using (\ref{1.2}), (\ref{3.1}),
(\ref{3.2}), (\ref{3.5}) and $\overline J({\ltr}(TM))={\Rad} TM$, we compute
\begin{equation}
\begin{array}{lll}
0=\overline g(W,\overline JX)=\overline g(\overline JW,X)=\overline g(PtW,X), \\
0=\overline g(W,\overline J\xi )=\overline g(\overline JW,\xi )=\overline g(LfW,\xi ), \\
0=\overline g(W,\overline JN)=\overline g(\overline JW,N)=\overline g(QtW,N).
\end{array}
\label{3.6}
\end{equation}
Since $\overline g$ is non-degenerate on $S(TM)$ and ${\Rad} TM\oplus {\ltr}(TM)$, from (\ref{3.6})
it follows that $PtW=LfW=QtW=0$. Then (\ref{3.5}) becomes $\overline JW=SfW$, which means that $S(TM^\bot )$ is
holomorphic.
\end{proof}

\begin{theorem}
Let $(\overline M,\overline J,\overline g,\overline {\widetilde g})$ be a $2n$-dimensional
almost complex manifold with Norden metric and $M$ be an $m$-dimensional submanifold of $\overline M$. The 
submanifold $(M,g)$ is a CR-submanifold with an $r$-dimensional totally real distribution $D^\bot $ iff 
$(M,\widetilde g)$ is an $r$-lightlike Radical transversal lightlike submanifold. 
\end{theorem}
\begin{proof}
As far as we know, CR-submanifolds of almost complex manifolds with Norden metric are not studied. Therefore,
first we give some preliminaries about them. 
\par 
Let $(M,g)$ be a CR-submanifold of $\overline M$ and assume that it is not generic.  
Hence, $(M,g)$ is endowed with two complementary orthogonal with respect to 
$\overline g$ distributions $D ({\dim} D=2p)$ and 
$D^\bot ({\dim} D^\bot =r: 1\leq r< {\min}(m,2n-m))$ such that $\overline JD=D, \, \overline JD^\bot \subset TM^\bot $.
Following \, \cite{Y-K} , for any $X\in \Gamma (TM)$ we have
\begin{equation}
\overline JX=TX+FX ,
\label{3.7}
\end{equation}
where $TX$ is the tangential part of $\overline JX$ and $FX$ is the normal part of $\overline JX$. Then $T$ is an
endomorphism on the tangent bundle $TM$ and $F$ is a normal bundle-valued 1-form on $TM$. Similarly, for any 
$V\in \Gamma (TM^\bot )$ we have
\begin{equation}
\overline JV=tV+fV ,
\label{3.8}
\end{equation}
where $tV$ is the tangential part of $\overline JV$ and $fV$ is the normal part of $\overline JV$. Then $f$ is an
endomorphism on the normal bundle $TM^\bot $ and $t$ is a tangent bundle-valued 1-form on $TM^\bot $. If we denote
by $P$ and $Q$ the projection morphisms of $TM$ on $D$ and $D^\bot $ respectively, for any $X\in \Gamma (TM)$ we
can write
\begin{equation}
X=PX+QX,
\label{3.9}
\end{equation} 
where $PX\in \Gamma (D)$ and $QX\in \Gamma (D^\bot )$.
Using (\ref{3.7}) and (\ref{3.8}), for $Y\in \Gamma (TM)$ and $U\in \Gamma (TM^\bot )$ we compute
$\overline g(\overline JX,Y)=\overline g(TX,Y)$ and $\overline g(\overline JV,U)=\overline g(fV,U)$. Since the
almost complex structure $\overline J$ is an anti-isometry with respect to the Norden metrics $\overline g$ and
$\overline {\widetilde g}$, we obtain that $T$ and $f$ are self-adjoint operators on $TM$ and
$TM^\bot $ with respect to both metrics. We note that $T$ and $f$ are skew self-adjoint \,
\cite{Y-K} when $\overline M$ is an almost Hermitian manifold. We also find $\overline g(FX,V)=\overline g(X,tV)$. 
Moreover, applying $\overline J$ to the equality (\ref{3.9}) we obtain
\begin{equation}
\overline JX=\overline JPX+\overline JQX,
\label{3.10}
\end{equation} 
where $\overline JPX\in \Gamma (D)$ and $\overline JQX\in \Gamma (\overline JD^\bot )$. From 
(\ref{3.7}) and (\ref{3.10}) it follows that $T=\overline JP$ and $F=\overline JQ$.
Due to the fact that $D^\bot $ is non-degenerate with respect to $\overline g$, there exists an orthonormal basis 
$\{\xi _1,\ldots,\xi _r\}$ of $D^\bot $, i.e.
\begin{equation}
g(\xi _i,\xi _i)=\epsilon _i, \, \, \epsilon _i=\pm 1; \quad g(\xi _i,\xi _j)=0, \, \, i\neq j; 
\quad i,j=1,2,\ldots,r. 
\label{3.11}
\end{equation}
Since ${\Ker }\overline J=\{0\}$, $\{\overline J\xi _1,\ldots,\overline J\xi _r\}$ is a 
basis of $\overline JD^\bot $. Moreover, it is an orthonormal basis such that
\begin{equation}
\overline g(\overline J\xi _i,\overline J\xi _i)=-\epsilon _i; \quad 
\overline g(\overline J\xi _i,\overline J\xi _j)=0, \, \, i\neq j; 
\quad i,j=1,2,\ldots,r. 
\label{3.12}
\end{equation}
Consequently, $\overline JD^\bot $ is an $r$-dimensional non-degenerate subbundle of $TM^\bot $ 
with respect to $\overline g$ and we put
\begin{equation}
TM^\bot =\overline JD^\bot \bot (\overline JD^\bot )^\bot ,
\label{3.13}
\end{equation}
where $(\overline JD^\bot )^\bot $ is the complementary orthogonal vector subbundle of $\overline JD^\bot $ in
$TM^\bot $. We denote by $P_1$ and $P_2$ the projection morphisms of $TM^\bot $ on $\overline JD^\bot $ and 
$(\overline JD^\bot )^\bot $, respectively. Then for any $V\in \Gamma (TM^\bot )$ we put
\begin{equation}
V=P_1V+P_2V,
\label{3.14}
\end{equation} 
where $P_1V\in \Gamma (\overline JD^\bot )$ and $P_2V\in \Gamma ((\overline JD^\bot )^\bot )$. Now, we will
show that the subbundle $(\overline JD^\bot )^\bot $ is holomorphic with respect to $\overline J$. Take
$W\in \Gamma ((\overline JD^\bot )^\bot )$, $X\in \Gamma (TM)$ and according to (\ref{3.10}) we compute $\overline g(\overline JW,X)=\overline g(W,\overline JX)=0$, which means that $\overline JW\in \Gamma (TM^\bot )$. On the other
hand, for each $N\in \Gamma (\overline JD^\bot )$ we have $\overline JN\in \Gamma (D^\bot )$, which implies 
$\overline g(\overline JW,N)=\overline g(W,\overline JN)=0$, i.e. 
$\overline JW\in \Gamma ((\overline JD^\bot )^\bot )$. We continue by proving that the normal bundle 
$TM^{\widetilde \bot }$ of the submanifold $(M,\widetilde g)$ coincides with the subbundle 
$D^\bot \bot (\overline JD^\bot )^\bot $ of $T\overline M$. First, let us suppose that 
$\overline Y\in \Gamma (D^\bot \bot (\overline JD^\bot )^\bot )$. Then for $\forall X\in \Gamma (TM)$ using 
(\ref{3.10}) we have $\overline {\widetilde g}(X,\overline Y)=\overline g(\overline JX,\overline Y)=0$. Hence, 
$\overline Y\in \Gamma (TM^{\widetilde \bot })$, i.e. the following relation holds
\begin{equation}
D^\bot \bot (\overline JD^\bot )^\bot \subseteq TM^{\widetilde \bot }. 
\label{3.15}
\end{equation}
Now, let $\overline Y\in \Gamma (TM^{\widetilde \bot })$. Then $\overline Y\in \Gamma (T\overline M)$ and 
$\overline {\widetilde g}(X,\overline Y)=0$ for $\forall X\in \Gamma (TM)$. The last is equivalent to 
$\overline g(X,\overline J\overline Y)=0$ which shows that $\overline J\overline Y \in \Gamma (TM^\bot )$. 
Since $(\overline JD^\bot )^\bot $ is holomorphic with respect to $\overline J $, applying $\overline J$ 
to (\ref{3.13}) we obtain $\overline J(TM^\bot )=D^\bot \bot (\overline JD^\bot )^\bot $. So, we conclude that 
$\overline Y \in \Gamma (D^\bot \bot (\overline JD^\bot )^\bot )$ and we have
\begin{equation}
TM^{\widetilde \bot }\subseteq D^\bot \bot (\overline JD^\bot )^\bot . 
\label{3.16}
\end{equation}
From (\ref{3.15}) and (\ref{3.16}) it follows that $TM^{\widetilde \bot }=D^\bot \bot (\overline JD^\bot )^\bot $.
Further, for $X\in \Gamma (D)$, $\xi \in \Gamma (D^\bot )$ and $W\in \Gamma ((\overline JD^\bot )^\bot )$, we get
$\overline {\widetilde g}(\xi ,X)=\overline g(\overline J\xi ,X)=0$, \, $\overline {\widetilde g}(\xi ,W)=
\overline g(\overline J\xi ,W)=0$. The last two equalities show that $D^\bot $ is orthogonal to $D$ and 
$(\overline JD^\bot )^\bot $ with respect to $\overline {\widetilde g}$. Then the following decompositions are valid
\begin{equation}
TM=D\widetilde {\bot }D^\bot , 
\label{3.17}
\end{equation}
\begin{equation}
TM^{\widetilde {\bot }}=D^\bot \widetilde {\bot }(\overline JD^\bot )^\bot  . 
\label{3.18}
\end{equation}
From (\ref{3.17}) and (\ref{3.18}) it is clear that the smooth distribution $D^\bot $ on $(M,\widetilde g)$ of
${\rank} \, r>0$ is an intersection of the tangent bundle $TM$ and the normal bundle $TM^{\widetilde {\bot }}$ of
$(M,\widetilde g)$. Hence, $(M,\widetilde g)$ is an $r$-lightlike submanifold of $\overline M$ with Radical
distribution ${\Rad} TM$ which coincides with $D^\bot $. According to (\ref{3.17}) and (\ref{3.18}), the distribution
$D$ and the vector bundle $(\overline JD^\bot )^\bot  $ are a screen distribution $S(TM)$ and a screen transversal
vector bundle $S(TM^{\widetilde {\bot }})$ of $(M,\widetilde g)$, respectively. The basis $\{\xi _1,\ldots,\xi _r\}$ of $D^\bot $ satisfying (\ref{3.11}) is a basis of ${\Rad} TM$ and consequently with respect to $\widetilde g$ we have
$\widetilde g(\xi _i,\xi _j)=0$ for any $i,j=1,\ldots ,r$. We put $N_i=-\epsilon _i\overline J\xi _i$, \, $i=1,\ldots ,r$. Then $\{N_1,\ldots,N_r\}$ is a basis of $\overline JD^\bot $ such that
\begin{equation}
\overline {\widetilde g}(N_i,N_j)=-\epsilon _i\epsilon _j\overline g(\overline J\xi _i,\xi _j)=0, 
\quad  i,j=1,\ldots ,r,
\label{3.19}
\end{equation}
\begin{equation}
\overline {\widetilde g}(N_i,\xi _i)=\epsilon _i\overline g(\xi _i,\xi _i)=1; \,\, 
\overline {\widetilde g}(N_i,\xi _j)=\epsilon _i\overline g(\xi _i,\xi _j)=0; \, i,j=1,\ldots ,r.
\label{3.20}
\end{equation}
From (\ref{3.19}) it follows that $\overline JD^\bot $ is a lightlike vector bundle with respect to $\overline {\widetilde g}$. The equalities (\ref{3.20}) show that $\overline JD^\bot $ is not orthogonal to ${\Rad} TM$ with
respect to $\overline {\widetilde g}$. It is easy also to check that $\overline JD^\bot $ is orthogonal to $S(TM)$
and $S(TM^{\widetilde {\bot }})$ with respect to $\overline {\widetilde g}$. So, the tangent bundle $T\overline M$
of $\overline M$ has the following decomposition
\begin{equation}
T\overline M=S(TM)\widetilde {\bot }S(TM^{\widetilde {\bot }})\widetilde {\bot }({\Rad} TM\widetilde {\oplus }
\overline JD^\bot ).
\label{3.21}
\end{equation}
The equality (\ref{3.21}) implies that $S(TM^{\widetilde {\bot }})\widetilde {\bot }({\Rad} TM\widetilde {\oplus }
\overline JD^\bot )$ is the complementary orthogonal vector bundle $S(TM)^{\widetilde {\bot }}$ to $S(TM)$ in 
$T\overline M$ with respect to $\overline {\widetilde g}$. Moreover, ${\Rad} TM\widetilde {\oplus }\overline JD^\bot $
is the complementary orthogonal vector bundle to $S(TM^{\widetilde {\bot }})$ in $S(TM)^{\widetilde {\bot }}$ with
respect to $\overline {\widetilde g}$, i.e.
\begin{equation}
S(TM^{\widetilde {\bot }})^{\widetilde {\bot }}={\Rad} TM\widetilde {\oplus }\overline JD^\bot .
\label{3.22}
\end{equation}
Finally, taking into account (\ref{3.22}) we conclude that the vector bundle $\overline JD^\bot $ is the
complementary (but not orthogonal with respect to $\overline {\widetilde g}$) vector bundle to ${\Rad} TM$ in
$S(TM^{\widetilde {\bot }})^{\widetilde {\bot }}$, i.e. $\overline JD^\bot $ is a lightlike transversal vector
bundle of $(M,\widetilde g)$ with respect to the pair 
$(S(TM)=D, \, \, S(TM^{\widetilde {\bot }})=(\overline JD^\bot )^\bot )$. Hence, $(M,\widetilde g)$ is an
$r$-lightlike Radical transversal lightlike submanifold.
\par
Conversely, let $(M,\widetilde g)$ be an $r$-lightlike Radical transversal lightlike submanifold of $\overline M$.
Take $X\in \Gamma (S(TM))$, $\xi \in \Gamma ({\Rad} TM)$, $N\in \Gamma ({\ltr}(TM))$, 
$W\in \Gamma (S(TM^{\widetilde {\bot }}))$ and using (\ref{1.2}),
(\ref{3.1}), (\ref{3.2}) and Proposition 3.1, we get
\begin{equation}
\begin{array}{lll}
\overline g(X,\xi )=-\overline {\widetilde g}(\overline JX,\xi )=0, \quad
\overline g(X,N)=-\overline {\widetilde g}(\overline JX,N)=0, \\
\overline g(\xi ,N)=-\overline {\widetilde g}(\overline J\xi ,N)=0, \quad
\overline g(W,X)=-\overline {\widetilde g}(\overline JW,X)=0, \\
\overline g(W,\xi )=-\overline {\widetilde g}(\overline JW,\xi )=0, \quad
\overline g(W,N)=-\overline {\widetilde g}(\overline JW,N)=0.
\end{array}
\label{3.23}
\end{equation}
The equalities (\ref{3.23}) show that the vector bundles $S(TM)$, ${\Rad} TM$, ${\ltr}(TM)$ and $S(TM^{\widetilde {\bot }})$
are mutually orthogonal with respect to $\overline g$. This fact and (\ref{1.2}) imply the following decompositions
\begin{equation}
\begin{array}{lll}
TM=S(TM)\bot {\Rad} TM, \quad {\tr}(TM)={\ltr}(TM)\bot S(TM^{\widetilde {\bot }}), \\
T\overline M=TM\bot {\tr}(TM).
\end{array}
\label{3.24}
\end{equation}
From the last equality of (\ref{3.24}) it follows that the normal bundle $TM^\bot $ of the submanifold $(M,g)$
coincides with the transversal vector bundle ${\tr}(TM)$ of $(M,\widetilde g)$. Hence, both $TM$ and $TM^\bot $ are
non-degenerate with respect to $\overline g$. The distribution $S(TM)$ is holomorphic and non-degenerate with 
respect to $\widetilde g$. According to Lemma 3.1, $S(TM)$ is non-degenerate with respect to $g$, too.
If we assume that the distribution ${\Rad} TM$ is degenerate with respect to $g$, then the first equality of
(\ref{3.24}) implies that $TM$ is degenerate, which is a contradiction. Consequently, $(M,g)$ is endowed with
two complementary orthogonal with respect to $g$ distributions $D^\bot ={\Rad} TM$ and $D=S(TM)$ satisfying 
(\ref{3.1}) and (\ref{3.2}), respectively. Since $TM^\bot ={\tr}(TM)$ and by using (\ref{3.24}) we have 
$TM^\bot =\overline JTM\bot S(TM^{\widetilde {\bot }})$. Hence, the distribution $D^\bot $ is totally real 
and moreover $S(TM^{\widetilde {\bot }})$ is the complementary orthogonal with respect to $\overline g$ vector bundle
of $\overline JTM$ in $TM^\bot $. So, we establish that $(M,g)$ is a CR-submanifold of $\overline M$. 
\end{proof}
Taking into account the definition of a generic submanifold and a coisotropic one, as an immediate consequence
from Theorem 3.2, we obtain
\begin{corollary}
Let $(\overline M,\overline J,\overline g,\overline {\widetilde g})$ be a $2n$-dimensional
almost complex manifold with Norden metric and $M$ be an $m$-dimensional submanifold of $\overline M$. The 
submanifold $(M,g)$ is generic iff $(M,\widetilde g)$ is a coisotropic Radical transversal lightlike submanifold.
\end{corollary}
Also we state
\begin{theorem}
Let $(\overline M,\overline J,\overline g,\overline {\widetilde g})$ be a $2n$-dimensional
almost complex manifold with Norden metric and $M$ be an $m$-dimensional 
submanifold of $\overline M$. Then
\begin{arabiclist}[(2)]
\item The submanifold $(M,g)$ $(1<m<n)$ is totally real  iff 
$(M,\widetilde g)$ is an isotropic submanifold such that $\overline J(TM)={\ltr}(TM)$.
\item The submanifold $(M,g)$ is Lagrangian iff $(M,\widetilde g)$ is a totally lightlike submanifold
such that $\overline J(TM)={\ltr}(TM)$.
\end{arabiclist}
\end{theorem}
\begin{proof}
An $m$-dimensional ($1<m<n$) submanifold $(M,g)$ of $\overline M$ is totally real if it is a CR-submanifold with a
holomorphic distribution $D=\{0\}$, i.e. $TM=D^\bot $. As $1<m<n$, $\overline JTM\subset TM^\bot $. On the other hand,
for an isotropic submanifold $S(TM)=\{0\}$.
By replacing $D$ and $S(TM)$ by zero in the proof of Theorem 3.2, we establish the truth of assertion $(1)$. 
Hence, we have the following direct decompositions of the tangent bundle $T\overline M$ of
$\overline M$
\begin{equation}
\begin{array}{ll}
T\overline M=TM\bot \overline JTM\bot (\overline JTM)^\bot , \, \,
T\overline M=(TM\widetilde \oplus {\ltr}(TM))\widetilde \bot S(TM^{\widetilde \bot}),
\end{array}
\label{3.27}
\end{equation}
where $\overline JTM={\ltr}(TM)$ and $(\overline JTM)^\bot =S(TM^{\widetilde \bot})$.
\par
A Lagrangian submanifold is a particular case of a CR-submanifold, too. Now, $D=\{0\}$
and $\overline JD^\bot =\overline JTM=TM^\bot $. Hence, $(\overline JTM)^\bot =\{0\}$. So, following the proof of
Theorem 3.2 and taking into account
that for a totally lightlike submanifold $S(TM)=S(TM^{\widetilde \bot})=\{0\}$, we verify that 
assertion $(2)$ is true. 
\end{proof}
\begin{remark}
From (\ref{2.1}) we have $\overline  g(X,Y)=-\overline {\widetilde g}(\overline JX,Y)$. Therefore we can
interchange the metrics $\overline g$ (resp. $g$) and $\overline {\widetilde g}$ (resp. $\widetilde g$) in
Theorem 3.2, Corollary 3.1 and Theorem 3.3.
\end{remark}

\section{Relations between the induced geometric objects on a totally real and an isotropic submanifold of almost
complex manifold with Norden metric}
Further, our purpose is  to establish relations between the induced geometric objects on the submanifolds 
introduced in the previous section. As a first step in this direction, we start with the study of the submanifolds 
from assertion $(1)$ of Theorem 3.3. In this section we consider an $m$-dimensional $(1<m<n)$ totally real
submanifold $(M,g)$ of a $2n$-dimensional almost complex manifold with Norden metric 
$(\overline M,\overline J,\overline g,\overline {\widetilde g})$. Then from Theorem 3.3 it follows that $(M,\widetilde g)$ is an isotropic submanifold of
$\overline M$ such that $\overline J(TM)={\ltr}(TM)$. From now on, we will denote: 
$\overline X,\overline Y,\overline Z,\overline U\in \Gamma (T\overline M)$; $X,Y,Z,U\in \Gamma (TM)$; 
$V,V^\prime \in \Gamma (TM^\bot )=\Gamma ({\tr}(TM))$; $N,N^\prime \in \Gamma (\overline JTM)=\Gamma ({\ltr}(TM))$
and $W,W^\prime \in \Gamma ((\overline JTM)^\bot )=\Gamma (S(TM^{\widetilde {\bot }}))$. 
\par
Let $\overline {\widetilde \nabla }$ be the
Levi-Civita connection of the metric $\overline {\widetilde g}$ on $\overline M$. According to \, \cite{D-B} , 
the formulas of Gauss and Weingarten for an $r$-lightlike submanifold  $(M,\widetilde g)$ of $\overline M$ are
\begin{equation}
\begin{array}{ll}
\overline {\widetilde \nabla }_XY=\widetilde \nabla _XY+\widetilde h(X,Y) , \\
\overline {\widetilde \nabla }_XV=-\widetilde A_VX+\nabla ^t_XV, 
\end{array}
\label{3.25}
\end{equation}
where: $\{\widetilde \nabla _XY, \widetilde A_VX\}$ and $\{\widetilde h(X,Y), \nabla ^t_XV\}$ belong to $\Gamma (TM)$
and $\Gamma ({\tr}(TM))$, respectively; $\widetilde \nabla$ and $\nabla ^t$ are linear connections on $(M,\widetilde g)$
and ${\tr}(TM)$, respectively. Moreover, $\widetilde \nabla$ is torsion-free linear connection, $\widetilde h$ is a 
$\Gamma ({\tr}(TM))$-valued symmetric ${\cal F}(M)$-bilinear form on $\Gamma (TM)$ and $\widetilde A$ is a 
$\Gamma (TM)$-valued ${\cal F}(M)$-bilinear form on $\Gamma ({\tr}(TM))\times \Gamma (TM)$. In general, 
$\widetilde \nabla$ and $\nabla ^t$ are not metric connections.
In \, \cite{D-B} the following objects are introduced
\begin{equation*}
\begin{array}{ll}
h^l(X,Y)=L(\widetilde h(X,Y)) ; \quad h^s(X,Y)=S(\widetilde h(X,Y)) ; \\
D^l_XV=L(\nabla ^t_XV) ; \quad \quad \quad D^s_XV=S(\nabla ^t_XV),
\end{array}
\end{equation*}
where $L$ and $S$ are the projection morphisms of ${\tr}(TM)$ on ${\ltr}(TM)$ and $S(TM^{\widetilde \bot })$, respectively.
Then the formulas (\ref{3.25}) become
\begin{equation*}
\begin{array}{ll}
\overline {\widetilde \nabla }_XY=\widetilde \nabla _XY+h^l(X,Y)+h^s(X,Y), \\
\overline {\widetilde \nabla }_XV=-\widetilde A_VX+D^l_XV+D^s_XV.
\end{array}
\end{equation*}
Besides $D^l$ and $D^s$ do not define linear connections on ${\tr}(TM)$ but they are Otsuki connections on ${\tr}(TM)$
with respect to $L$ and $S$, respectively. We note that for an $r$-lightlike submanifold $D^l$ is a metric
Otsuki connection on ${\tr}(TM)$.
Now, the formulas of Gauss and Weingarten for an isotropic submanifold  $(M,\widetilde g)$ of $\overline M$ are
\begin{equation}
\begin{array}{lll}
\overline {\widetilde \nabla }_XY=\widetilde \nabla _XY+h^s(X,Y) , \\
\overline {\widetilde \nabla }_XN=-\widetilde A_NX+\nabla ^l_XN+{\cal D}^s(X,N) , \\
\overline {\widetilde \nabla }_XW=-\widetilde A_WX+{\cal D}^l(X,W)+\nabla ^s_XW , 
\end{array}
\label{3.26}
\end{equation}
where: $\widetilde \nabla$ is a metric linear connection on $(M,\widetilde g)$;
$\nabla ^l$ and $\nabla ^s$, defined by $\nabla ^l_XN=D^l_XN$ and $\nabla ^s_XW=D^s_XW$, 
are metric linear connections on ${\ltr}(TM)$ and $S(TM^{\widetilde \bot })$,
respectively; ${\cal D}^l$ and ${\cal D}^s$, defined by ${\cal D}^l(X,W)=D^l_XW$ and ${\cal D}^s(X,N)=D^s_XN$,
are ${\cal F}(\overline M)$-bilinear mappings. By using (\ref{3.26}) and taking into account that 
$\overline {\widetilde \nabla }$ is a metric connection we obtain
\begin{equation*}
\begin{array}{ll}
\overline {\widetilde g}(h^s(X,Y),W)+\overline {\widetilde g}(Y,{\cal D}^l(X,W))=0; \quad
\overline {\widetilde g}({\cal D}^s(X,N),W)=\overline {\widetilde g}(N,\widetilde A_WX); \\
\overline {\widetilde g}(\widetilde A_NX,N^\prime )+\overline {\widetilde g}(\widetilde A_{N^\prime }X,N)=0.
\end{array}
\end{equation*} 
Further, we will find the formulas of Gauss and Weingarten for $(M,g)$. Denoting by 
$\overline \nabla $ and $\nabla $ the Levi-Civita connections of $\overline g$ and $g$ on $\overline M$ 
and $(M,g)$ respectively, we have
\begin{equation}
\begin{array}{ll}
\overline \nabla _XY=\nabla _XY+h(X,Y) , \\
\overline \nabla _XV=-A_VX+D_XV ,
\end{array}
\label{3.28}
\end{equation}
where: $h$ is the second fundamental form of $(M,g)$; $A_V$ is the shape operator of $(M,g)$ with
respect to $V$; $D$ is the normal connection on $TM^\bot $ which is a metric linear connection. Taking into account
the first equality of (\ref{3.27}), the formulas (\ref{3.28}) become
\begin{equation}
\begin{array}{ll}
\overline \nabla _XY=\nabla _XY+h^1(X,Y)+h^2(X,Y) , \\
\overline \nabla _XV=-A_VX+D^1_XV+D^2_XV ,
\end{array}
\label{3.29}
\end{equation}
where we denote
\begin{equation*}
\begin{array}{ll}
h^1(X,Y)=P_1(h(X,Y)) ; \quad h^2(X,Y)=P_2(h(X,Y)) ; \\
D^1_XV=P_1(D_XV) ; \quad \quad \quad D^2_XV=P_2(D_XV)
\end{array}
\end{equation*}
and $P_1$, $P_2$ are the projection morphisms of $TM^\bot $ on $\overline JTM$, $(\overline JTM)^\bot $, 
respectively.
By direct calculations we check that $D^1$ and $D^2$ do not define linear connections on $TM^\bot $ but
both are Otsuki connections with respect to the vector bundle morphism $P_1$ and $P_2$, respectively. So, the
formulas (\ref{3.29}) can be written as follows
\begin{equation}
\begin{array}{lll}
\overline \nabla _XY=\nabla _XY+h^1(X,Y)+h^2(X,Y) , \\
\overline \nabla _XN=-A_NX+\nabla ^1_XN+{\cal D}^2(X,N) , \\
\overline \nabla _XW=-A_WX+{\cal D}^1(X,W)+\nabla ^2_XW , 
\end{array}
\label{3.30}
\end{equation}
where: $\nabla ^1$ and $\nabla ^2$, defined by $\nabla ^1_XN=D^1_XN$ and $\nabla ^2_XW=D^2_XW$, are metric
linear connections on $\overline JTM$ and $(\overline JTM)^\bot $, respectively; ${\cal D}^1$ and
${\cal D}^2$, defined by ${\cal D}^1(X,W)=D^1_XW$ and ${\cal D}^2(X,N)=D^2_XN$, are ${\cal F}(\overline M)$-bilinear
mappings. As $\overline \nabla $ is a metric connection, from (\ref{3.30}) we obtain
\begin{equation}
\overline g(h^1(X,Y),N)=\overline g(A_NX,Y),
\label{3.31}
\end{equation}
\begin{equation}
\overline g(h^2(X,Y),W)=\overline g(A_WX,Y),
\label{3.32}
\end{equation}
\begin{equation}
\overline g({\cal D}^2(X,N),W)=-\overline g({\cal D}^1(X,W),N).
\label{3.33}
\end{equation}
\begin{theorem}
Let $(M,g)$ be a totally real submanifold of $(\overline M,\overline J,\overline g,\overline {\widetilde g})$. 
Then the following assertions are equivalent:
\begin{romanlist}[iv]
\item $D^1$ is a metric Otsuki connection on $TM^\bot $.
\item ${\cal D}^1(X,W)=0$, for any $X\in \Gamma (TM)$, $W\in \Gamma ((\overline JTM)^\bot )$.
\item ${\cal D}^2(X,N)=0$, for any $X\in \Gamma (TM)$, $N\in \Gamma (\overline JTM)$.
\item $D^2$ is a metric Otsuki connection on $TM^\bot $.
\end{romanlist}
\end{theorem}
\begin{proof}
Using that $\nabla ^1$ and $\nabla ^2$ are metric linear connections on $\overline JTM$ and $(\overline JTM)^\bot $ respectively and (\ref{3.33}) we have
\begin{equation*}
\begin{array}{ll}
(D^1_X\overline g)(N,N^\prime )=(D^2_X\overline g)(N,N^\prime )=0,\, \,
(D^1_X\overline g)(W,W^\prime )=(D^2_X\overline g)(W,W^\prime )=0, \\
(D^1_X\overline g)(W,N)=-\overline g({\cal D}^1(X,W),N)=\overline g({\cal D}^2(X,N),W)=-(D^2_X\overline g)(W,N), 
\end{array}
\end{equation*}
In this way we see that assertions $(i)$, $(ii)$, $(iii)$ and $(iv)$ are equivalent.
\end{proof}
According to (\ref{3.27}), we can write (\ref{3}) in the following manner
\begin{equation}
\overline {\widetilde \nabla }_{\overline X}\overline Y=\overline \nabla _{\overline X}\overline Y+
\Phi ^\prime (\overline X,\overline Y)+\Phi ^1(\overline X,\overline Y)+\Phi ^2(\overline X,\overline Y),
\label{3.34}
\end{equation}
where we denote by $\Phi ^\prime (\overline X,\overline Y)$, $\Phi ^1(\overline X,\overline Y)$ and
$\Phi ^2(\overline X,\overline Y)$ the part of $\Phi (\overline X,\overline Y)$ which belongs to $\Gamma (TM)$,
$\Gamma (\overline JTM)=\Gamma ({\ltr}(TM))$ and $\Gamma ((\overline JTM)^\bot )=\Gamma (S(TM^{\widetilde \bot}))$,
respectively. By using (\ref{3.26}) and (\ref{3.30}) from (\ref{3.34}) we get
\begin{equation*}
\begin{array}{llllll}
\widetilde \nabla _XY+h^s(X,Y)=\nabla _XY+h^1(X,Y)+h^2(X,Y)\\
\qquad \qquad \qquad \qquad +\Phi ^\prime (X,Y)+\Phi ^1(X,Y)+\Phi ^2(X,Y) ; \\
-\widetilde A_NX+\nabla ^l_XN+{\cal D}^s(X,N)=-A_NX+\nabla ^1_XN+{\cal D}^2(X,N) \\
\qquad \qquad \qquad \qquad \qquad \qquad \, \, \,  +\Phi ^\prime (X,N)+\Phi ^1(X,N)+\Phi ^2(X,N) ; \\
-\widetilde A_WX+{\cal D}^l(X,W)+\nabla ^s_XW=-A_WX+{\cal D}^1(X,W)+\nabla ^2_XW \\
\qquad \qquad \qquad \qquad \qquad \qquad \quad +\Phi ^\prime (X,W)+\Phi ^1(X,W)+\Phi ^2(X,W).
\end{array}
\end{equation*}
Comparing the parts belonging to
$\Gamma (TM)$, $\Gamma (\overline JTM)=\Gamma ({\ltr}(TM))$ and $\Gamma ((\overline JTM)^\bot )=\Gamma (S(TM^{\widetilde \bot}))$ of both sides of the above equations, we obtain the following relations
\begin{equation}
\begin{array}{lll}
\widetilde \nabla _XY=\nabla _XY+\Phi ^\prime (X,Y), \quad h^s(X,Y)=h^2(X,Y)+\Phi ^2(X,Y), \\
h^1(X,Y)=-\Phi ^1(X,Y), \quad  \quad \widetilde A_NX=A_NX-\Phi ^\prime (X,N), \\
\nabla ^l_XN=\nabla ^1_XN+\Phi ^1(X,N), \quad {\cal D}^s(X,N)={\cal D}^2(X,N)+\Phi ^2(X,N), \\ 
\widetilde A_WX=A_WX-\Phi ^\prime (X,W), \quad {\cal D}^l(X,W)={\cal D}^1(X,W)+\Phi ^1(X,W), \\
\nabla ^s_XW=\nabla ^2_XW+\Phi ^2(X,W).
\end{array}
\label{3.35}
\end{equation}
In \, \cite{GB} the eight classes of almost complex manifolds with Norden metric are characterized by
conditions for the tensor $F$. In \, \cite{GGM} the following relations between the tensor $F$ and $\Phi $ are
obtained 
\begin{equation}
\Phi (\overline X,\overline Y,\overline Z)=\displaystyle{
\frac{1}{2}\left(F(\overline J\overline Z,\overline X,\overline Y)-F(\overline X,\overline Y,\overline JZ)
-F(\overline Y,\overline JZ,\overline X)\right)},
\label{3.36}
\end{equation}
\begin{equation}
F(\overline X,\overline Y,\overline Z)=\Phi (\overline X,\overline Y,\overline J\overline Z)+
\Phi (\overline X,\overline Z,\overline J\overline Y),
\label{3.37}
\end{equation}
where $\Phi (\overline X,\overline Y,\overline Z)=\overline g(\Phi (\overline X,\overline Y),\overline Z)$.
Also in \, \cite{GGM} , by using (\ref{3.36}) and
(\ref{3.37}), every basic class of almost complex manifolds with Norden metric is characterized by conditions
for the tensor $\Phi $. Taking into account both types of characterization conditions, we can specify the  
equations (\ref{3.35}) when the ambient manifold $\overline M$ belongs to any basic class. Further, in this section
we consider only the case when $\overline M$ is a Kaehler manifold with Norden metric. Then the  characterization 
condition of $\overline M$ is $F(\overline X,\overline Y,\overline Z)=0$ which is equivalent to
$(\overline \nabla _{\overline X}\overline J)\overline Y=0$.
\begin{lemma}
Let $\overline M$ be a Kaehler manifold with Norden metric and $(M,g)$ be a totally real submanifold of $\overline M$.
Then 
\begin{equation}
h^1(X,Y)=0, \qquad A_{\overline JY}X=0,
\label{3.38}
\end{equation} 
\begin{equation}
\nabla ^1_X\overline JY=\overline J(\nabla _XY), 
\label{3.39}
\end{equation}
\begin{equation}
{\cal D}^2(X,\overline JY)=\overline J(h^2(X,Y)), 
\label{3.40}
\end{equation}
\begin{equation}
A_WX=\overline J({\cal D}^1(X,\overline JW)), 
\label{3.41}
\end{equation}
\begin{equation}
\nabla ^2_X\overline JW=\overline J(\nabla ^2_XW).
\label{3.42}
\end{equation}
\end{lemma}
\begin{proof}
From $\overline \nabla _{\overline X}\overline JY=\overline J\overline \nabla _{\overline X}\overline Y$ 
and (\ref{3.30}) we obtain the relations (\ref{3.39}) $\div $ (\ref{3.42}) and 
$h^1(X,Y)=\overline JA_{\overline JY}X$. On the other hand, replacing $N$ by $\overline JY$ in (\ref{3.31})
we have $h^1(X,Y)=-\overline JA_{\overline JY}X$. Hence, (\ref{3.38}) is fulfilled.
\end{proof}
The characterization condition $F=0$ for a Kaehler manifold with Norden $\overline M$
and (\ref{3.36}) imply that $\Phi =0$ for $\overline M$. 
Then the formulas (\ref{3}) and (\ref{3.35}) become
\begin{equation}
\overline \nabla _XY=\overline {\widetilde \nabla }_XY
\label{3.43}
\end{equation}
and
\begin{equation}
\begin{array}{lll}
\widetilde \nabla _XY=\nabla _XY, \quad h^s(X,Y)=h^2(X,Y), \quad h^1(X,Y)=0,\\
\widetilde A_NX=A_NX, \quad \nabla ^l_XN=\nabla ^1_XN, \quad {\cal D}^s(X,N)={\cal D}^2(X,N),\\ 
\widetilde A_WX=A_WX, \quad {\cal D}^l(X,W)={\cal D}^1(X,W), \quad
\nabla ^s_XW=\nabla ^2_XW.
\end{array}
\label{3.44}
\end{equation}
From (\ref{3.44}) it is clear that the relations (\ref{3.38}) $\div $ (\ref{3.42}) are valid if we replace
$\nabla , h^1, h^2, A_N, A_W, \nabla ^1, \nabla ^2, {\cal D}^1, {\cal D}^2$ by $\widetilde \nabla , h^l, h^s,
\widetilde A_N, \widetilde A_W, \nabla ^l, \nabla ^s, {\cal D}^l, {\cal D}^s$, respectively.
\begin{theorem}
Let $\overline M$ be a Kaehler manifold with Norden metric and $(M,g)$ be a totally real submanifold of $\overline M$.
Then the following assertions are equivalent:
\begin{romanlist}[v]
\item $(M,g)$ is totally geodesic.
\item $h^2$ vanishes identically on $(M,g)$.
\item $A_W$ vanishes identically on $(M,g)$.
\item ${\cal D}^1(X,W)=0$, for $\forall X\in \Gamma (TM), \forall W\in \Gamma ((\overline JTM)^\bot )$.
\item ${\cal D}^2(X,N)=0$, for $\forall X\in \Gamma (TM), \forall N\in \Gamma (\overline JTM)$.
\end{romanlist}
\end{theorem}
\begin{proof}
The equality (\ref{3.38}) implies $h(X,Y)=h^2(X,Y)$ and consequently $(i)\Longleftrightarrow (ii)$. The equivalence
of $(ii)$ and $(iii)$ follows from (\ref{3.32}). Because of (\ref{3.40}) and (\ref{3.41}) we have 
$(ii)\Longleftrightarrow (v)$ and $(iii)\Longleftrightarrow (iv)$, respectively.
\end{proof}
\begin{theorem}
Let $\overline M$ be a Kaehler manifold with Norden metric. 
Then the following assertions are equivalent:
\begin{romanlist}[vi]
\item $(M,g)$ is totally geodesic.
\item $(M,\widetilde g)$ is totally geodesic.
\item $D^1$ is a metric Otsuki connection on $TM^\bot $.
\item $D^2$ is a metric Otsuki connection on $TM^\bot $.
\item $\nabla ^t$ is a metric linear connection on ${\tr}(TM)$.
\item $D^s$ is a metric Otsuki connection on ${\tr}(TM)$.
\end{romanlist}
\end{theorem}
\begin{proof}
According to Theorem 4.1 and Theorem 4.2, each from the assertions $(i), (iii), (iv)$ is equivalent to the
condition ${\cal D}^2(X,N)=0$. In \cite{D-B} (page 159 and 167) it is proved that assertions $(v), (vi)$ and $(ii)$
are equivalent to ${\cal D}^s(X,N)=0$ and $h^s(X,Y)=0$, respectively. From (\ref{3.43}) we have 
${\cal D}^2(X,N)={\cal D}^s(X,N)=0$ and $h^2(X,Y)=h^s(X,Y)=0$. Finally, according to Theorem 4.2 the conditions
${\cal D}^2(X,N)=0$ and $h^2(X,Y)=0$ are equivalent which completes the proof.
\end{proof}
{\bf Structure equations of the submanifolds $(M,g)$ and $(M,\widetilde g)$}.
\par
Denote by $\overline R, R$ and $R^\bot $  the curvature tensors of type $(1,3)$ of $\overline \nabla $, $\nabla $ and
$D$, respectively. We define the curvature tensors $(R^\bot )^1$ and $(R^\bot )^2$ of $\nabla ^1$ and $\nabla ^2$ by 
$(R^\bot )^1(X,Y,N)=\nabla ^1_X\nabla ^1_YN-\nabla ^1_Y\nabla ^1_XN-\nabla ^1_{[X,Y]}N$ and
$(R^\bot )^2(X,Y,W)=\nabla ^2_X\nabla ^2_YW-\nabla ^2_Y\nabla ^2_XW-\nabla ^2_{[X,Y]}W$. Then by straightforward
calculations using (\ref{3.30}) and Lemma 4.1 we obtain
\begin{equation}
\begin{array}{l}
\overline R(X,Y,Z)=R(X,Y,Z)-A(h^2(Y,Z),X)+A(h^2(X,Z),Y) \\
\qquad \qquad \quad +\overline J\left(A(\overline Jh^2(Y,Z),X)-A(\overline Jh^2(X,Z),Y)\right) \\
\qquad \qquad \quad +(\nabla _Xh^2)(Y,Z)-(\nabla _Yh^2)(X,Z),
\end{array}
\label{3.45}
\end{equation}
\begin{equation}
\begin{array}{l}
\overline R(X,Y,W)=(R^\bot )^2(X,Y,W)-h^2(X,A_WY)+h^2(Y,A_WX)\\
\qquad \qquad \quad \, \, +\overline J\left(h^2(X,A_{\overline JW}Y)-h^2(Y,A_{\overline JW}X)\right)-(\nabla _XA)(W,Y)\\
\qquad \qquad \, \,  +(\nabla _YA)(W,X)+\overline J\left((\nabla _XA)(\overline JW,Y)-(\nabla _YA)(\overline JW,X)\right),
\end{array}
\label{3.46}
\end{equation}
where 
\begin{equation*}
\begin{array}{ll}
(\nabla _Xh^2)(Y,Z)=\nabla ^2_Xh^2(Y,Z)-h^2(\nabla _XY,Z)-h^2(Y,\nabla _XZ), \\
(\nabla _XA)(W,Y)=\nabla _XA(W,Y)-A\left(\nabla ^2_XW,Y\right)-A\left(W,\nabla _XY\right).
\end{array}
\end{equation*}
For convenience we denote $A_VX$ by $A(V,X)$ where it is necessary. The equality 
$\overline \nabla _{\overline X}\overline JY=\overline J\overline \nabla _{\overline X}\overline Y$ implies
\begin{equation}
\overline R(\overline X,\overline Y,\overline J\overline Z)=
\overline J\overline R(\overline X,\overline Y,\overline Z).
\label{3.46'}
\end{equation}
Hence, from (\ref{3.45}) we have
\begin{equation}
\begin{array}{l}
\overline R(X,Y,\overline JZ)=\overline JR(X,Y,Z)-\overline J\left(A(h^2(Y,Z),X)-A(h^2(X,Z),Y)\right) \\
\qquad \qquad \quad -A(\overline Jh^2(Y,Z),X)+A(\overline Jh^2(X,Z),Y) \\
\qquad \qquad \quad +\overline J\left((\nabla _Xh^2)(Y,Z)-(\nabla _Yh^2)(X,Z)\right).
\end{array}
\label{3.47}
\end{equation}
Now, using $D_XN=D^1_XN+D^2_XN=\nabla ^1_XN+{\cal D}^2(X,N)$, \,  $D_XW=D^1_XW+D^2_XW=
{\cal D}^1(X,W)+\nabla ^2_XW$ and Lemma 4.1 we compute
\begin{equation}
\begin{array}{l}
R^\bot (X,Y,\overline JZ)=D_XD_Y\overline JZ-D_YD_X\overline JZ-D_{[X,Y]}\overline JZ \\
\qquad \qquad \qquad \, =\overline JR(X,Y,Z)+\overline J\left((\nabla _Xh^2)(Y,Z)-(\nabla _Yh^2)(X,Z)\right)\\
\qquad \qquad \qquad \, -\overline J\left(A(h^2(Y,Z),X)-A(h^2(X,Z),Y)\right),
\end{array}
\label{3.48}
\end{equation}
\begin{equation}
\begin{array}{l}
R^\bot (X,Y,W)=D_XD_YW-D_YD_XW-D_{[X,Y]}W \\
\qquad \quad  =(R^\bot )^2(X,Y,W)+\overline J\left((\nabla _XA)(\overline JW,Y)-(\nabla _YA)
(\overline JW,X)\right)\\
\qquad \quad  -\overline J\left(h^2(Y,A(\overline JW,X))-h^2(X,A(\overline JW,Y))\right).
\end{array}
\label{3.49}
\end{equation}
From (\ref{3.43}) and (\ref{3.44}) it follows $\overline R=\overline {\widetilde R},R=\widetilde R,
(R^\bot )^1=R^l, (R^\bot )^2=R^s$. Thus the equalities (\ref{3.45}), (\ref{3.46}), (\ref{3.46'}), (\ref{3.47}) 
are valid if we replace $\overline R, R, (R^\bot )^2, A_W, h^2$ by $\overline {\widetilde R}, \widetilde R, R^s,
\widetilde {A}_W, h^s$, respectively. We note that $\overline {\widetilde R}(X,Y,Z),
\overline {\widetilde R}(X,Y,\overline JZ), \overline {\widetilde R}(X,Y,W)$ can be obtained from \cite{D-B}
\, (page 175) by using  (\ref{3.44}) and Lemma 4.1. Taking into account (\ref{3.39}) and (\ref{3.42}) we have
\begin{equation}
\overline JR(X,Y,Z)=(R^\bot )^1(X,Y,\overline JZ)=R^l(X,Y,\overline JZ)=\overline J\widetilde R(X,Y,Z)
\label{3.50}
\end{equation}
and
\begin{equation*}
\overline J(R^\bot )^2(X,Y,W)=(R^\bot )^2(X,Y,\overline JW)=R^s(X,Y,\overline JW)=\overline JR^s(X,Y,W).
\end{equation*}
It is known that the Levi-Civita connection $\overline \nabla $ (resp. $\overline {\widetilde \nabla }$) is said
to be flat if $\overline R=0$ (resp. $\overline {\widetilde R}=0$). Also, if $R^\bot =0$ then the normal connection
$D$ is said to be flat. Analogously, we say that $\nabla ^1, \nabla ^2, \nabla ^l, \nabla ^s$ are flat if the
corresponding curvature tensors $(R^\bot )^1, (R^\bot )^2, R^l, R^s$ are equal to zero. As a corollary of 
(\ref{3.50}) we state
\begin{proposition}
For the submanifolds $(M,g)$ and $(M,\widetilde g)$ of a Kaehler manifold with Norden metric $\overline M$ the
following assertions are equivalent:
$1) \nabla $ is flat; \quad 
$2) \widetilde \nabla $ is flat; \quad
$3) \nabla ^1$ is flat; \quad
$4) \nabla ^l$ is flat.
\end{proposition}
From (\ref{3.45}) we obtain the equation of Gauss and Codazzi
\begin{equation*}
\overline R(X,Y,Z,U)=R(X,Y,Z,U)-\overline g\left(A(h^2(Y,Z),X),U\right)+\overline g\left(A(h^2(X,Z),Y),U\right)
\end{equation*}
and
\begin{equation*}
\begin{array}{l}
(\overline R(X,Y,Z))^\bot =(\nabla _Xh^2)(Y,Z)-(\nabla _Yh^2)(X,Z)\\
\qquad \qquad \qquad \, \, +\overline J\left(A(\overline Jh^2(Y,Z),X)-
A(\overline Jh^2(X,Z),Y)\right).
\end{array}
\end{equation*}
The equality (\ref{3.46'}) and the fact that $\overline J$ is an anti-isometry with respect to $\overline g$ imply
$\overline R(\overline X,\overline Y,\overline Z,\overline U)=\overline g\left(\overline R(\overline X,\overline Y,\overline Z),\overline U\right)$
is a Kaehler tensor, i.e.
\begin{equation}
\overline R(\overline X,\overline Y,\overline J\overline Z,\overline J\overline U)=
-\overline R(\overline X,\overline Y,\overline Z,\overline U).
\label{3.51}
\end{equation}
Using (\ref{3.51}), (\ref{3.48}), (\ref{3.49}) and (\ref{3.32}) for the equation of Ricci we get
\begin{equation*}
\begin{array}{lll}
\overline g\left(\overline R(X,Y,\overline JZ),\overline JU\right)=-\overline R(X,Y,Z,U), \\
\overline g\left(\overline R(X,Y,\overline JZ),W\right)=\overline g\left((\nabla _Xh^2)(Y,Z)-
(\nabla _Yh^2)(X,Z),\overline JW\right),\\
\overline g\left(\overline R(X,Y,W),W^\prime \right)=\overline g\left((R^\bot )^2(X,Y,W),W^\prime \right)-
\overline g\left(A(\overline JW^\prime ,Y),A(\overline JW,X)\right)\\
+\overline g\left(A(\overline JW^\prime ,X),A(\overline JW,Y)\right)+\overline g\left(A_WX,A_{W^\prime }Y\right)-
\overline g\left(A_{W^\prime }X,A_WY\right).
\end{array}
\end{equation*}
For the curvature tensor $\overline {\widetilde R}$ of type $(0,4)$ we have
\begin{equation}
\overline {\widetilde R}(\overline X,\overline Y,\overline Z,\overline U)=
\overline {\widetilde g}(\overline {\widetilde R}(\overline X,\overline Y,\overline Z),\overline U)=
\overline R(\overline X,\overline Y,\overline Z,\overline J\overline U).
\label{3.52}
\end{equation}
By direct calculations we check that $\overline {\widetilde R}$ is a Kaehler tensor, too. The structure equations
of $\overline {\widetilde R}$ we obtain using (\ref{3.32}), (\ref{3.44}), (\ref{3.52}) and equations of Gauss, 
Codazzi and Ricci
\begin{equation*}
\begin{array}{lllll}
\overline {\widetilde R}(X,Y,Z,N)=\overline {\widetilde g}\left(\widetilde R(X,Y,Z),N\right)-
\overline {\widetilde g}\left(\widetilde A(h^s(Y,Z),X),N\right) \\
\qquad \qquad \qquad \, \, +\overline {\widetilde g}\left(\widetilde A(h^s(X,Z),Y),N\right), \\
\overline {\widetilde R}(X,Y,N^\prime ,N)=\overline {\widetilde g}\left(\widetilde
A(\overline Jh^s(Y,\overline JN),X),N^\prime \right)
-\overline {\widetilde g}\left(\widetilde 
A(\overline Jh^s(X,\overline JN),Y),N^\prime \right), \\
\overline {\widetilde R}(X,Y,W,N)=\overline {\widetilde g}\left((\widetilde {\nabla }_Y\widetilde A)(W,X),N\right)
-\overline {\widetilde g}\left((\widetilde {\nabla }_X\widetilde A)(W,Y),N\right),\\
\overline {\widetilde R}(X,Y,Z,U)=\overline {\widetilde g}\left(h^s(Y,U),h^s(X,Z)\right)-
\overline {\widetilde g}\left(h^s(X,U),h^s(Y,Z)\right), \\
\overline {\widetilde R}(X,Y,W,U)=\overline {\widetilde g}\left((\widetilde {\nabla }_Yh^s)(X,U),W\right)
-\overline {\widetilde g}\left((\widetilde {\nabla }_Xh^s)(Y,U),W\right),\\
\overline {\widetilde R}(X,Y,W^\prime ,W)=\overline {\widetilde g}\left(R^s(X,Y,W^\prime ),W\right)-
\overline {\widetilde g}\left(\widetilde A_WY,\overline J\widetilde A_{\overline JW^\prime }X\right)\\
\qquad \quad +\overline {\widetilde g}\left(\widetilde A_WX,\overline J\widetilde A_{\overline JW^\prime }Y\right)-
\overline {\widetilde g}\left(\widetilde A_{W^\prime }X,\overline J\widetilde A_{\overline JW}Y\right)+
\overline {\widetilde g}\left(\widetilde A_{W^\prime }Y,\overline J\widetilde A_{\overline JW}X\right), \\
\overline {\widetilde g}\left(\widetilde R(X,Y,Z),N\right)+\overline {\widetilde g}\left(R^l(X,Y,N),Z\right)=0.
\end{array}
\end{equation*}

\section{Examples}
According to the obtained results in Section 3, each from the constructed examples below is an example for a
submanifold of an almost complex manifold with Norden metric which is non-degenerate with respect to the one
Norden metric and lightlike with respect to the other.
\begin{example}
We consider the Lie group $GL(2;\mathbb{R})$ with a Lie algebra $gl(2;\mathbb{R})$. 
The real Lie algebra $gl(2;\mathbb{R})$ is spanned by
the left invariant vector fields $\{X_1,X_2,X_3,X_4\}$ where we set
\[
\hspace*{-4mm}
X_1=\left(\begin{array}{ll}
1 & 0
\cr 0 & 0 
\end{array}\right) , \, \,
X_2=\left(\begin{array}{ll}
0 & 1
\cr 0 & 0 
\end{array}\right) , \, \,
X_3=\left(\begin{array}{ll}
0 & 0
\cr 1 & 0 
\end{array}\right) , \, \,
X_4=\left(\begin{array}{ll}
0 & 0
\cr 0 & 1 
\end{array}\right) .
\]
We define an almost complex structure $\overline J$ and a left invariant metric $\overline g$ on $gl(2;\mathbb{R})$ by
\begin{equation}
\overline JX_1=X_4, \quad \overline JX_2=X_3, \quad \overline JX_3=-X_2, \quad \overline JX_4=-X_1
\label{3.53'}
\end{equation}
and
\begin{equation}
\begin{array}{ll}
\overline g(X_i,X_i)=-\overline g(X_j,X_j)=-1; \quad i=1,3; \quad j=2,4;\\
\overline g(X_i,X_j)=0; \quad i\neq j; \quad i,j=1,2,3,4.
\end{array}
\label{3.53''}
\end{equation}
Using (\ref{3.53'}) and (\ref{3.53''}) we check that the metric $\overline g$ is a Norden metric and consequently
$(GL(2;\mathbb{R}),\overline J,\overline g,\overline {\widetilde g})$ is a 4-dimensional almost complex
manifold with Norden metric. The real special linear group $SL(2;\mathbb{R})=\{A\in GL(2;\mathbb{R}): det(A)=1\}$ 
is a Lie subgroup of $GL(2;\mathbb{R})$ with a Lie algebra $sl(2;\mathbb{R})$ of all $(2\times 2)$ real traceless matrices. Thus
$SL(2;\mathbb{R})$ is a 3-dimensional submanifold of $GL(2;\mathbb{R})$ and $sl(2;\mathbb{R})$ is a subalgebra of
$gl(2;\mathbb{R})$ spanned by $\{X_1-X_4, X_2, X_3\}$. We find that the normal space $sl(2;\mathbb{R})^\bot $
is spanned by $\{X_1, X_4\}$. Hence $sl(2;\mathbb{R})\cap sl(2;\mathbb{R})^\bot =\{X_1-X_4\}$, i.e. 
$(SL(2;\mathbb{R}),g)$ is an 1-dimensional lightlike submanifold with ${\Rad} sl(2;\mathbb{R})=
span\left\{\xi =\displaystyle{\frac{X_1-X_4}{\sqrt{2}}}\right\}$. As ${\dim} {\Rad} sl(2;\mathbb{R})=
{\codim} SL(2;\mathbb{R})=1$ it follows that $(SL(2;\mathbb{R}),g)$ is a coisotropic submanifold of $GL(2;\mathbb{R})$.
The screen distribution $S(sl(2;\mathbb{R}))$ is spanned by $\{X_2, X_3=\overline JX_2\}$ which means that
$S(sl(2;\mathbb{R}))$ is holomorphic with respect to $\overline J$. Choose $sl(2;\mathbb{R})^\bot =span\left\{\xi , 
\displaystyle{\frac{-X_1-X_4}{\sqrt{2}}}\right\}$ we have ${\ltr}(sl(2;\mathbb{R}))=span\left\{ 
\displaystyle{\frac{-X_1-X_4}{\sqrt{2}}}\right\}$. Using (\ref{3.53'}) we check that $\overline J{\Rad} sl(2;\mathbb{R})
=\overline J\xi ={\ltr}(sl(2;\mathbb{R}))$. Thus $(SL(2;\mathbb{R}),g)$ is a coisotropic Radical transversal lightlike
submanifold of $GL(2;\mathbb{R})$ and from Corollary 3.1 it follows that $(SL(2;\mathbb{R}),\widetilde g)$ is a
generic submanifold.

\end{example}
Now, we consider the Lie group $GL(2;\mathbb{C})$ with its real 8-dimensional Lie algebra $gl(2;\mathbb{C})$. We 
define a complex structure $\overline J$ on $gl(2;\mathbb{C})$ by $\overline J\overline X=i\overline X$ for any left invariant vector field $\overline X\in gl(2;\mathbb{C})$. Hence we have $[\overline J\overline X,\overline Y]
=\overline J[\overline X,\overline Y]$, i.e.
$\overline J$ is a bi-invariant complex structure. The Lie algebra $gl(2;\mathbb{C})$ is spanned by
the left invariant vector fields $\{X_1,X_2,\ldots ,X_8\}$ where we set
\[
\hspace*{-4mm}
X_1=\left(\begin{array}{ll}
1 & 0
\cr 0 & 0 
\end{array}\right) , \, \,
X_3=\left(\begin{array}{ll}
0 & 1
\cr 0 & 0 
\end{array}\right) , \, \,
X_5=\left(\begin{array}{ll}
0 & 0
\cr 1 & 0 
\end{array}\right) , \, \,
X_7=\left(\begin{array}{ll}
0 & 0
\cr 0 & 1 
\end{array}\right) ,
\]
$X_2=\overline JX_1, \qquad X_4=\overline JX_3, \qquad X_6=\overline JX_5, \qquad X_8=\overline JX_7$.
\par
We define a left invariant metric $\overline g$ on $gl(2;\mathbb{C})$ by
\begin{equation}
\begin{array}{ll}
\overline g(X_i,X_i)=-\overline g(X_j,X_j)=1; \quad i=2,3,5,8; \quad j=1,4,6,7; \\
\overline g(X_i,X_j)=0; \quad i\neq j; \quad i,j=1,2,\ldots ,8.
\end{array}
\label{3.53}
\end{equation}
It is easy to check that the so defined metric $\overline g$ is a Norden metric on $gl(2;\mathbb{C})$. Thus 
$(GL(2;\mathbb{C}),\overline J,\overline g,\overline {\widetilde g})$ is an 8-dimensional almost complex
manifold with Norden metric. Since the metric $\overline g$ is left invariant, for the Levi-Civita
connection $\overline \nabla $ of $\overline g$ we have
\begin{equation}
2\overline g(\overline \nabla _{\overline X}\overline Y,\overline Z)=\overline g
([\overline X,\overline Y],\overline Z)+\overline g([\overline Z,\overline X],\overline Y)+
\overline g([\overline Z,\overline Y],\overline X)
\label{3.54}
\end{equation}
for any $\overline X, \overline Y, \overline Z\in gl(2;\mathbb{C})$.
Using (\ref{3.54}) and $\overline J$ is bi-invariant we get $F(\overline X,\overline Y,\overline Z)=0$. 
Consequently, $(GL(2;\mathbb{C}),\overline J,\overline g,\overline {\widetilde g})$ is a Kaehler manifold
with Norden metric.
\begin{example}
The unitary group $U(2)=\left\{A\in GL(2;\mathbb{C}): A^{-1}=\overline A^T\right\}$ is a Lie group which is a Lie
subgroup of $GL(2;\mathbb{C})$. Hence $U(2)$ is a submanifold of $GL(2;\mathbb{C})$. The real Lie algebra
$u(2)$ which consists of all $(2\times 2)$ skew Hermitian matrices, i.e.
$u(2)=\left\{A\in gl(2;\mathbb{C}): \overline A+A^T=0\right\}$ is the Lie algebra of $U(2)$ and it is a 
4-dimensional subalgebra of $gl(2;\mathbb{C})$. We have $u(2)=span\left\{F_1=X_2, \, \,F_2=X_3-X_5, \, \,F_3=
X_4+X_6, \, \, F_4=X_8\right\}$. Denote by $g$ and $\widetilde g$ the induced metrics on $u(2)$ 
of $\overline g$ and $\overline {\widetilde g}$, respectively.
Using (\ref{3.53}) we calculate $g(F_1,F_1)=g(F_4,F_4)=1$; \, $g(F_2,F_2)=-g(F_3,F_3)=2$; \,  
$g(F_i,F_j)=0, \, i\neq j$, $i,j=1,2,3,4$. Therefore $(U(2),g)$ is a 4-dimensional non-degenerate submanifold of $GL(2;\mathbb{C})$. Moreover $g$ is a Lorentz metric. We find that the normal space $u(2)^\bot $ is spanned by 
$\{\overline JF_1,\overline JF_2,\overline JF_3,\overline JF_4\}$ which means that $(U(2),g)$ is a Lagrangian
submanifold of $GL(2;\mathbb{C})$. From Theorem 3.3 it follows that $(U(2),\widetilde g)$ is a totally lightlike 
submanifold of $GL(2;\mathbb{C})$ such that $\overline J({\Rad} u(2))=\overline Ju(2)={\ltr}(u(2))$.
\end{example}

\begin{example}
The special unitary group $SU(2)=U(2)\bigcap SL(2;\mathbb{C})=\{A\in GL(2;\mathbb{C}): A^{-1}=\overline  A ^T, \, {\det}(A)=1\}$ is a Lie group which is a Lie subgroup of $GL(2;\mathbb{C})$. Hence $SU(2)$ is a submanifold of $GL(2;\mathbb{C})$. The real Lie algebra $su(2)$ which consists of all $(2\times 2)$ traceless skew Hermitian  
matrices, i.e. $su(2)=\{A\in gl(2;\mathbb{C}): \overline A+A^T=0, \, \, {\tr} A=0\}$ is the Lie algebra of $SU(2)$
and it is a 3-dimensional subalgebra of $gl(2;\mathbb{C})$. We have $su(2)=span\left\{\xi _1=X_2-X_8, \, \,
\xi _2=X_3-X_5, \, \,\xi _3=X_4+X_6\right\}$. Using (\ref{3.53}) we calculate 
$g(\xi _1,\xi _1)=g(\xi _2,\xi _2)=-g(\xi _3,\xi _3)=2$, \, 
$g(\xi _i,\xi _j)=0, \, i\neq j$, $i,j=1,2,3$. Therefore $(SU(2),g)$ is a 3-dimensional non-degenerate submanifold of $GL(2;\mathbb{C})$. Moreover $g$ is a Lorentz metric. We find that the normal space $su(2)^\bot $ is spanned by 
$\left\{\overline J\xi _2,\overline J\xi _3,X_1,X_7,X_2+X_8\right\}$. As $\overline J\xi _1=-X_1+X_7$ belongs to $su(2)^\bot $
it follows that $\overline Jsu(2)\subset su(2)^\bot $. Hence $(SU(2),g)$ is a totally real submanifold
of $GL(2;\mathbb{C})$. The left invariant vector fields $\xi _4=X_1+X_7$ and $\overline J\xi _4=X_2+X_8$ belong to
$su(2)^\bot $. By direct calculations we check that $\overline g(\overline J\xi _i,\overline J\xi _j)=
\overline g(\overline J\xi _i,\xi _4)=0, \, \, i\neq j, \, \, i,j=1,2,3,4$. Thus $\left\{\overline J\xi _1,
\overline J\xi _2,\overline J\xi _3,\xi _4,\overline J\xi _4\right\}$ is an orthogonal basis of $su(2)^\bot $
with respect to $\overline g$. It is clear that the complementary orthogonal vector subspace $(\overline Jsu(2))^\bot $ of $\overline Jsu(2)$ in $su(2)^\bot $ is spanned by $\left\{\xi _4,\overline J\xi _4\right\}$ and consequently
it is holomorphic. Now, from Theorem 3.3 it follows that $(SU(2),\widetilde g)$ is an isotropic submanifold of
$GL(2;\mathbb{C})$ such that $\overline J({\Rad} su(2))=\overline Jsu(2)={\ltr}(su(2))$ and 
$S(su(2)^{\widetilde \bot })=(\overline Jsu(2))^\bot $.
\end{example}


\begin{thebibliography}{1}

\bibitem{D-B} K. L. Duggal and A. Bejancu, {\it Lightlike Submanifolds of Semi-Riemannian Manifolds and Applications},
Kluwer Academic, 364, 1996.        

\bibitem{D-S} K. L. Duggal and B. Sahin, {\it Differential Geometry of Lightlike Submanifolds}, Birkh\"auser Verlag AG,
2010.

\bibitem{GB} G. Ganchev and A. Borisov, {\it Note on almost complex manifolds with Norden metric}, C. R. Acad.
Bulg. Sci. 39 (1986), pp 31-34.


\bibitem{GGM} G. Ganchev, K. Gribachev and V. Mihova, {\it B-connections and their conformal invariants on 
conformally Kaehler manifolds with B-metric}, Publications De L'Institut Mathematique, Nouvelle serie, tome 42 (56),
1987, pp 107-121.

\bibitem{GGM_2} G. Ganchev, K. Gribachev, V. Mihova, \emph{Holomorphic hypersurfaces of Kaehler manifolds with Norden metric}. Plovdiv Univ. Sci. Works - Math., vol. 23, no. 2, (1985), pp 221–236. (in Bulgarian)        
     
     
\bibitem{MM1} M. Manev.
\emph{Classes of real time-like hypersurfaces of a Kaehler manifold with B-metric}. J. Geom., vol. 75, no. 1-2 (2002), pp 113-122. 


\bibitem{MM2} M. Manev,
\emph{Classes of real isotropic hypersurfaces of a Kaehler manifold with B-metric}. Compt. rend. Acad. bulg.
Sci., vol. 55, no. 4 (2002), pp 27-32.
        
\bibitem{Sahin} B. Sahin, {\it Transversal lightlike submanifolds of indefinite Kaehler manifolds}, Analele 
Universitatii de Vest, Timisoara, Seria Matematica - Informatica, XLIV, 1, (2006), pp 119-145.
        


\bibitem{Y-K} K. Yano and M. Kon, {\it CR-Submanifolds of Kaehlerian and Sasakian Manifolds}, Birkh\"auser, 1983.
        
        


       
\end{thebibliography}
\end{document}